\documentclass[11pt]{article}
\title{Topological obstruction for Anosov representations into $\PO(p, q)$}
	\author{Kewei Li}
	\date{}

\usepackage[utf8x]{inputenc}
\usepackage[T1]{fontenc}
\usepackage{lmodern}         %

\usepackage{cite}

\usepackage{natbib} 
\usepackage{array}

\usepackage[a4paper,top=3cm,bottom=2cm,left=3cm,right=3cm,marginparwidth=1.75cm]{geometry}

\usepackage{tikz-cd}

\usepackage{geometry}
\usepackage{graphicx}
\usepackage{pifont}
\usepackage{amsfonts}
\usepackage{mathrsfs}
\usepackage{amsmath}
\usepackage{amssymb}
\usepackage{extarrows}
\usepackage{soul}
\usepackage{enumitem}
\usepackage{amsthm}

\usepackage{microtype}

\makeatletter
\newsavebox{\@brx}
\newcommand{\llangle}[1][]{\savebox{\@brx}{\(\m@th{#1\langle}\)}%
  \mathopen{\copy\@brx\kern-0.5\wd\@brx\usebox{\@brx}}}
\newcommand{\rrangle}[1][]{\savebox{\@brx}{\(\m@th{#1\rangle}\)}%
  \mathclose{\copy\@brx\kern-0.5\wd\@brx\usebox{\@brx}}}
\makeatother

\usepackage{amsthm,amsmath,amssymb}

\usepackage{verbatim}

\usepackage{url}
\usepackage[colorinlistoftodos]{todonotes}
\usepackage[colorlinks=true, allcolors=blue]{hyperref}

\numberwithin{equation}{subsection}

\usepackage{amsthm}  
\theoremstyle{fancy}
\usepackage{environ,tikz}

\let\caron\v

\newcommand{\bbN}{\mathbb{N}}

\newcommand{\bbZ}{\mathbb{Z}}

\newcommand{\bR}{\mathbb{R}}

\newcommand{\bP}{\mathbb{P}}

\newcommand{\calC}{\mathcal{C}}

\newcommand{\fa}{\mathfrak{a}}

\newcommand{\fg}{\mathfrak{g}}
\newcommand{\fk}{\mathfrak{k}}
\newcommand{\fh}{\mathfrak{h}}
\newcommand{\fu}{\mathfrak{u}}

\newcommand{\fp}{\mathfrak{p}}
\newcommand{\fq}{\mathfrak{q}}

\renewcommand{\so}{\mathfrak{so}}

\newcommand{\gl}{\mathfrak{gl}}

\newcommand{\li}{\limits}

\newcommand{\Ker}{\operatorname{Ker}}

\newcommand{\Hom}{\operatorname{Hom}}

\renewcommand{\det}{\operatorname{det}}

\newcommand{\Span}{\operatorname{Span}}

\newcommand{\Graph}{\operatorname{Graph}}

\newcommand{\Stab}{\operatorname{Stab}}

\newcommand{\Ad}{\operatorname{Ad}}
\newcommand{\ad}{\operatorname{ad}}

\renewcommand{\v}[1]{\vert #1 \vert}
\newcommand{\V}[1]{\Vert #1 \Vert}
\newcommand{\inp}[1]{\l\langle #1 \r\rangle}

\renewcommand{\t}{\tilde}

\newcommand{\inv}{^{-1}}

\newcommand{\varep}{\varepsilon}
\newcommand{\Over}{\overline}
\renewcommand{\O}{\overline}
\newcommand{\rtar}{\rightarrow}

\newcommand{\hrtar}{\hookrightarrow}

\newcommand{\pa}{\partial}
\newcommand{\al}{\alpha}
\newcommand{\be}{\beta}
\renewcommand{\th}{\theta}
\newcommand{\Th}{\Theta}

\newcommand{\La}{\Lambda}
\newcommand{\ga}{\gamma}
\newcommand{\Ga}{\Gamma}

\newcommand{\de}{\delta}
\newcommand{\De}{\Delta}

\renewcommand{\O}{\operatorname{O}}

\newcommand{\PO}{\operatorname{PO}}
\newcommand{\SO}{\operatorname{SO}}

\newcommand{\SU}{\operatorname{SU}}

\newcommand{\PGL}{\operatorname{PGL}}
\newcommand{\opnm}{\operatorname}
\newcommand{\G}{\operatorname{G}}

\newcommand{\Ein}{\operatorname{Ein}}
\newcommand{\IsGr}{\operatorname{IsGr}}
\newcommand{\CoIsGr}{\operatorname{CoIsGr}}

\newcommand{\ob}{\operatorname{b}}
\newcommand{\og}{\operatorname{g}}

\renewcommand{\l}{\left}
\renewcommand{\r}{\right}

\newcommand{\mbf}{\mathbf}

\newtheorem{theo}{Theorem}[section]
\newtheorem{lem}[theo]{Lemma}
\newtheorem{coro}[theo]{Corollary}

\newtheorem{prop}[theo]{Proposition}

\newtheorem*{crit*}{Criterion}

\usepackage{float}

\newtheorem*{maintheo}{Main Theorem}

\theoremstyle{definition}
\newtheorem{defn}[theo]{Definition}
\newtheorem{rem}[theo]{Remark}

\newtheorem*{theo*}{Theorem}
\newtheorem*{lem*}{Lemma}

\usepackage{mathpazo} 
\usepackage[T1]{fontenc}

\usepackage{etoolbox}
\AtBeginEnvironment{thebibliography}{\small}

\begin{document}

	\maketitle
	
	\begin{abstract}
		We establish new restrictions on the virtual cohomological dimension of word hyperbolic groups admitting Anosov representations into $\PO(p, q)$ ($p > q \geq 2$) or the exceptional Lie group $\G_2'$. Our main result states that, if $P_k$ is the parabolic subgroup of $\PO(p, q)$ stabilizing an isotropic $k$-plane in $\bR^{p, q}$, a group admitting a $\{P_1,P_q\}$-Anosov representation into $\PO(p, q)$ has virtual cohomological dimension strictly less than $p$. As a result, a $\mbf{H}^{p, q-1}$-convex cocompact subgroup of $\PO(p, q)$ in the sense of \cite{DGK1}, cannot act properly discontinuously and  cocompactly on $\mbf{H}^{p, q-1}$. The main theorem also implies that a group admitting a $P_{\al}$-Anosov representation into the exceptional Lie group $\G_2'$ has virtual cohomological dimension strictly less than $4$, where $\al$ is a short restricted root of $\G_2'$. As an application, we prove that the homogeneous space $\G_2'/ \SU(2, 1)$ does not admit any compact quotient.
	\end{abstract}

	\section{Introduction}

	Anosov representations were first introduced by Labourie in his celebrated paper \cite{Lab} to provide a geometric interpretation of representations in the Hitchin components. Thanks to the work of many authors (see, for instance,  Guichard--Wienhard \cite{GW}, Kapovich--Leeb--Porti \cite{KLP, KLP18},  Guéritaud--Guichard--Kassel--Wienhard \cite{GGKW}, Bochi--Potrie--Sambarino \cite{BPS}),  the theory of Anosov representations was quickly developed. Anosov subgroups are one of the most important classes of discrete subgroups of Lie groups, whose study is related to higher rank Teichmüller theory, geometric structures, divisible convex sets, and compact quotients of homogeneous spaces. 
	
	Roughly speaking, a morphism $\rho$ from a hyperbolic group $\Ga$ to real semisimple Lie group $G$ is called a \textit{$P$-Anosov representation}, where $P$ is a parabolic subgroup, if the action of $\Ga$ on $G/P$ exhibits some hyperbolic behavior. Classifying triples $(\Ga, G, P)$ such that $\Ga$ admits a $P$-Anosov representation into $G$ has become an important question in the domain (see, for instance, \cite{RK, Kon, Dey, DGR, PT, ER, KiTr, Bla}).
	
	In this paper, we prove a new topological obstruction for the existence of such triples $(\Ga, G, P)$. For $p, q$ integers, let $G$ be the Lie group $\PO(p, q)$ and for any $k \leq \min\{p, q\}$, denote by $P_k$ the parabolic subgroup stabilizing an isotropic $k$-plane in $\bR^{p, q}$.
	
	\begin{maintheo} \label{theo po}
		Let $p > q \geq 2$ be two integers and $\Ga$ be a word hyperbolic group. If $\rho: \Ga \rtar \PO(p, q)$ is $\{P_1, P_q\}$-Anosov, then the virtual cohomological dimension of $\Ga$ is strictly less than $p$.
	\end{maintheo}

	Observe that the assumption $p > q$ is necessary: in \cite{HDRMD} Monclair constructed a $\{P_1, P_{4}\}$-Anosov representation into $\PO(4, 4)$ from a uniform lattice of $\SU(2, 1)$. 
	
		\paragraph{The exceptional Lie group $\G_2'$.}	
		
	A direct application of our main theorem concerns Anosov representations into the exceptional Lie group $\G_2'$. The group $\G_2'$ is the group of automorphism of the split octonion algebra. It is a rank $2$ subgroup of $\PO(4, 3)$ whose maximal parabolic subgroups $P_\al$ and $P_\be$ are respectively associated to the short and long restricted root. Since $P_\al$-Anosov representations into $\G_2'$ give a $\{P_1, P_3\}$-Anosov representations into $\PO(4, 3)$ (see \cite{DE}), our main theorem gives:

		\begin{coro} \label{theo g2}
			Let $\Ga$ be a discrete group admitting $P_\al$-Anosov representation in $\G_2'$, where $\al$ is a short restricted root, then the virtual cohomological dimension of $\Ga$ is at most $3$. 
		\end{coro}
		
		So far, no example of such representations is known for the case $\opnm{vcd} (\Ga) = 3$. We believe that there are no such representations.

		Together, the main theorem and Corollary~\ref{theo g2} give partial answers to \cite[Question 0.20]{ER} for $\G_2'$ and $\PO(4, 3)$.

		\paragraph{Compact Clifford--Klein form.} Let $\mbf{X}$ be a reductive $G$-homogeneous space, that is, $\mbf{X}$ can be realized as the quotient of a real reductive Lie group $G$ by a reductive subgroup $H$. A \textit{compact Clifford--Klein form} (also called \textit{compact quotient}) is a pair $(\mbf{X}, \Ga)$ where $\Ga$ is a discrete subgroup of $G$ acting properly discontinuously and cocompactly on $\mbf{X}$.

		When $\mbf{X}$ is not Riemannian (equivalently, when $H$ is non-compact), an infinite discrete subgroup $\Ga$ of $G$ might not act properly discontinuous on $\mbf{X}$ (for instance, if $\Ga$ is a subgroup of $H$). Benoist \cite{Ben} and Kobayashi \cite{Ko96} independently  proved a criterion for a subgroup $\Ga$ of $G$ to act properly discontinuously on $X$. As a result, when $G$ and $H$ have the same real rank, only finite groups act properly discontinuously on $\mbf{X}$.
		
		In this paper, the reductive $\PO(p, q)$-homogeneous space we consider is the \textit{pseudo-Riemannian hyperbolic space}
	\[\mbf{H}^{p, q-1} \simeq  \PO(p, q)/  \bP \l(\O(p, q-1) \times \O(1) \r)\ .\]
Historically, the first class of compact quotient to be extensively studied was the case of $\mbf{H}^{2, 1}$ (see, for instance, \cite{Gold, KR, Kli, Sa, Mess, Ka1, GK, GKW, DT, Tho1, Tho2}).
		
		The recent work of Kassel and Tholozan \cite{KT} link the theory of compact quotients with the theory of Anosov representations. A consequence of their work is 
		\begin{crit*}[Kassel--Tholozan\cite{KT}] 
			The pair $(\mbf{H}^{p, q-1}, \Ga)$ defines a compact quotient, if and only if, $p \geq q$ and $\Ga $ is $P_q$-Anosov subgroup of $\PO(p, q)$ of virtual cohomological dimension $p$.
		\end{crit*}

	More recently, Kassel, Morita, and Tholozan \cite{KMT} obtained topological obstructions for the existence of compact quotient of $\mbf{H}^{p, q-1}$ whenever $(p, q)$ satisfies a numerical condition. Within the range not excluded by their obstruction, the smallest dimensional case is $\mbf{H}^{4,2}$: it is conjectured by Kobayashi in \cite{Kob97, Kob01} that $\mbf{H}^{4, 2}$ does not admit compact quotient.
	
	The pseudo-Riemannian hyperbolic space $\mbf{H}^{4, 2}$ admits another realization as a reductive $\G_2'$-homogeneous space: the Lie group $\G_2'$ acts transitively on $\mbf{H}^{4, 2}$ with stabilizer conjugated to $\SU(2, 1)$. As a direct consequence of Corollary~\ref{theo g2} and Kassel--Tholozan criterion, we obtain:
	\begin{coro} \label{prop po43}
		The $\G_2'$-homogeneous space $\G_2'/ \SU(2, 1)$ does not admit a compact quotient. 
	\end{coro}

	\paragraph{Incompatibility result with $\mbf{H}^{p, q-1}$-convex cocompactness.}
	In \cite{DGK1} the authors generalize the notion of convex cocompactness to the case of $\mbf{H}^{p, q-1}$. They proved that $\mbf{H}^{p, q-1}$-convex cocompact subgroups of $\PO(p, q)$ are exactly $P_1$-Anosov subgroups whose limit sets satisfy a negativity condition.
		
		\begin{coro} \label{coro A}
		Let $p > q \geq 2$ be two integers. A discrete subgroup of $\PO(p, q)$ cannot simultaneously be $\mbf{H}^{p, q-1}$-convex cocompact and define a compact quotient of $\mbf{H}^{p, q-1}$.
		\end{coro}
	
	This result was first observed in the study of $\mbf{H}^{2, 1}$. Around 2022, Kassel, Morita and Tholozan conjectured that this phenomenon holds more generally. This was communicated to the author by Morita and Tholozan in April 2026.

		For the proof of our main theorem, we formalize the following lemma concerning about the topology of the Gromov boundary, which is also of independent interest.		
		
		\begin{lem} \label{connected G/H}
			Let $G/ H$ be a connected reductive homogeneous spaces, $K_H$ a maximal compact subgroup of $H$ and $K$ a maximal compact subgroup of $G$ containing $K_H$. Assume that 
			\[\dim(G/H) - \dim(K/K_H) > 1\ .\]
If a word hyperbolic subgroup $\Ga < G$ defines a compact quotient of $G/ H$, then the Gromov boundary $\pa_\infty \Ga$ of $\Ga$ is connected.
		\end{lem}

		The assumption ``word hyperbolic'' is automatically satisfied in corank-$1$ case by \cite[Prop.~1.13]{KT}, consequently, we have the following corollary.
		\begin{coro} \label{connected}
			Let $p\geq q \geq 2$ be two integers. If $\rho: \Ga \rtar \PO(p, q)$ defines a compact quotient of $\mbf{H}^{p, q-1}$, then the Gromov boundary $\pa_\infty \Ga$ of $\Ga$ is connected. 
		\end{coro}

		
		\paragraph{Outline of the proof of the main Theorem.}
					
		Suppose, toward a contradiction, that $\opnm{vcd}(\Ga)=p$. The $P_q$-Anosov condition and Kassel--Tholozan criterion imply that $\Gamma$ defines a compact quotient of $\mbf{H}^{p, q-1}$. Corollary~\ref{connected}  shows that $\pa_\infty\Ga$ is connected. Since $p>q$, the $P_1$-Anosov assumption and \cite[Coro.~11.10]{DGK2} imply that the representation is $\mbf{H}^{p, q-1}$-convex cocompact. Its limit set in the boundary is therefore negative; maximal dimension makes it a topological $(p-1)$-sphere. The negativity of the limit set induces a strictly $1$-Lipschitz map $f:\mbf{S}^{p-1}\rtar \mbf{S}^{q-1}$ by \cite{SST, BK}. 
		
		The key observation is that such a representation is $(1, 1, p)$-hyperconvex in the sense of Pozzetti--Sambarino--Wienhard \cite{PSW} and Zhang--Zimmer \cite{ZZ}. As a result, we prove that the map $f$ is a $\calC^1$-submersion and hence is surjective.  The surjectivity contradicts the strict $1$-Lipschitz property.

		\paragraph{Organization of the paper.}
		Section~\ref{Pre} is devoted to give some preliminaries required in our paper. First we recall some basics of Lie algebra and a characterization of Anosov representation. Then we recall some results about compact quotient of homogeneous space and prove Lemma~\ref{connected G/H}. Finally we recall the definition of pseudo-Riemannian hyperbolic space $\mbf{H}^{p, q-1}$ and the notion of $\mbf{H}^{p, q-1}$-convex cocompactness. After the preliminaries, we give the proof of our main theorem.

		\paragraph{Acknowledgements.}
		
		The author is profoundly grateful to their advisor, Jérémy Toulisse, for his constant guidance and unwavering support. Special thanks are due to Yosuke Morita and Nicolas Tholozan for their valuable input in April 2026. The author thanks Fanny Kassel for her interest in our work. The author also thanks Siwei Liang, who pointed out that Coro.~\ref{connected} could be generalized to Lem.~\ref{connected G/H}. The author also thanks Max Riestenberg for stimulating discussions about Lie algebra. In addition, the author benefited from the workshops held in Hefei (April 2026) and Hohenwart (August 2026), as well as a visit to the Max Planck Institute for Mathematics in the Sciences in Leipzig (May 2026), and thanks the respective organizers for their hospitality. The author acknowledges the support of the Institut Universitaire de France through the fellowship of Jérémy Toulisse.
		
		\paragraph{Declaration of AI use.}
		
		During the preparation, the author utilized OpenAI’s ChatGPT and Google’s Gemini for clarifying background concepts, exploratory proof assistance, and language polishing. Regarding proof assistance, ChatGPT-5.6 was consulted for brainstorming ideas on Coro.~\ref{connected}; although the model's suggested proof was mathematically incorrect, the interaction provided heuristic insights that helped the author formulate the proof of Coro.~\ref{connected} and further generalize the argument to establish Lem.~\ref{connected G/H}. Additionally, ChatGPT-5.6 assisted in identifying relevant background literature concerning the example of $\SU(2, 1)$.

	All core proof strategies and essential arguments were substantially developed and written by the author, with minimal assistance from AI tools. The author has thoroughly verified all derivations and assumes full responsibility for the entirety of this work.

		\section{Preliminaries and preparatory results} \label{Pre}
		
		The purpose of this section is to recall basic definitions and properties of Lie algebras, Anosov representations, compact quotients of reductive homogeneous spaces and $\mbf{H}^{p, q-1}$-convex cocompactness. We record the proof of Lemma~\ref{connected G/H} in Section~\ref{tech lem}.
				
		\subsection{Basics of Lie algebras} \label{Lie}
		We follow closely the exposition in \cite{GGKW, KT, KMT}. Let $G$ be a noncompact connected real linear reductive Lie group with a Cartan involution $\th$. The fixed-point subgroup $K:= G^\th = \{g \in G: \th(g) = g\}$ is a maximal compact subgroup of $G$. We use the same symbol $\th$ for the involution on $\fg$ induced by the Cartan involution. We have the Cartan decomposition $\fg = \fk \oplus \fp$, where 
			\[\fk = \{X\in \fg: \th(X) = X\}\ , \quad  \fp := \{X \in \fg: \th(X) = - X\}\ .\]
Note that $\fk$ is the Lie algebra of $K$ defined above.

A Cartan subspace $\fa \subset \fp$ is a maximal abelian subspace of $\fp$. For each linear functional $\al \in \fa^*$, define the restricted root space
		\[\fg_{\alpha} := \{X\in \fg: \ad_H X = \al(H) X, \ \forall H \in \fa\} \ .\]
The set of restricted roots is 
		\[\Sigma := \{ \al \in \fa^* \setminus \{0\}: \fg_{\alpha} \neq 0\}\ .\]
This gives the restricted root space decomposition $\fg = \fg_0 \oplus \bigoplus\li_{\al \in \Sigma} \fg_{\al}$, where $\fg_0$ is the centralizer of $\fa$ in $\fg$.

		Choose a system $\Sigma^+ \subset \Sigma$ of positive restricted roots, so that $\Sigma = \Sigma^+ \sqcup (-\Sigma^+)$. Let $\De \subset \Sigma^+$ be the corresponding set of simple restricted roots. Every positive restricted root is then a nonnegative integral linear combination of elements of $\De$. The associated closed positive Weyl chamber is
		\[\Over{\fa}^+ := \{Y \in \fa: \al(Y) \geq 0, \  \forall \ \al \in \Sigma^+\}\ .\]

		Given a subset of simple restricted roots $\Th \subset \De$, define $P_\Th$ and $P_\Th^-$ to be the normalizers in $G$ of the Lie algebra
		\[\fu_\Th = \bigoplus\li_{\al \in \Sigma_\Th^+} \fg_\al  \quad \text{and} \quad  \fu_\Th^- = \bigoplus\li_{\al \in \Sigma_\Th^+} \fg_{-\al}  \ ,\]
where $\Sigma_{\Th}^+ = \Sigma^+ \setminus  \Span (\De \setminus \Th)$. The group $P_\Th$ and $P_{\Th}^-$ are parabolic subgroups of $G$ with Lie algebra
		\[\fg_0 \oplus  \bigoplus\li_{\al \in \Sigma^+} \fg_\al \oplus  \bigoplus\li_{\al \in \Sigma^+ \setminus \Sigma_\Th^+} \fg_{-\al} \quad \text{and} \quad \fg_0 \oplus  \bigoplus\li_{\al \in \Sigma^+} \fg_{-\al} \oplus  \bigoplus\li_{\al \in \Sigma^+ \setminus \Sigma_\Th^+} \fg_{\al} \ .\]
In particular $P_{\varnothing} = G$ and $P_\De$ is a minimal parabolic subgroup of $G$.

		Recall the Cartan decomposition $G = K (\exp \Over{\fa}^+) K$ as follows: for any $g\in G$, it can be written as $g = k (\exp \mu(g)) k'$ for some $k, k' \in K$ and a unique $\mu(g) \in \Over{\fa}^+$. The resulting map
		\[\mu: G \rtar \Over{\fa}^+, \quad g \mapsto \mu(g)\ ,\]
is called the \textit{Cartan projection}.

		\begin{defn} \label{para}
		We say that two parabolic subgroups $x$ and $y$ of $G$ are \textit{compatible} (resp. \textit{transverse}) if their intersection is parabolic (resp. reductive).
		\end{defn}

		\paragraph{Example $\PO_0(p, q)$ for $p > q$ two positive integers.}
		
		\begin{defn}
			We fix integers $p, q \geq 1$ and consider a symmetric bilinear form $\ob$ of signature $(p, q)$ on $\bR^{p+q}$. We denote by $\bR^{p, q}$ the space $\bR^{p+q}$ endowed with $\ob$.
		
			For a subspace $V \subset \bR^{p, q}$, we say $V$ is \textit{isotropic} if $V \subset V^\perp$; we say $V$ is \textit{coisotropic} if $V^\perp$ is isotropic. Denote by $\IsGr_k \bR^{p, q}$ the set of isotropic $k$-planes in $\bR^{p, q}$ and by $\CoIsGr_{p+q-k} \bR^{p, q}$ the set of coisotropic $(p+q-k)$-planes in $\bR^{p, q}$. 
			
		\end{defn}

By definition, there is a natural isomorphism from $\IsGr_k \bR^{p, q}$ to $\CoIsGr_{p+q -k} \bR^{p, q}$ by sending $V$ to $V^\perp$. Consider an \textit{ordered Witt basis} of $\bR^{p, q}$, namely, 
			\[(e_1, \ldots, e_q, l_1, \ldots, l_{p-q}, f_q, \ldots, f_1) \ ,\]
satisfying $\ob(e_i, f_j) = \de_{ij}$, $\ob(e_i, e_j) =  \ob(f_i, f_j) = \ob(e_i, l_s) = \ob(f_i, l_s) = 0$ and $\ob(l_s, l_t) = \de_{st}$ where $1\leq i, j \leq q$, $1\leq s, t \leq p-q$. Define, for $1\leq k \leq q$,
			\[E_k := \Span\{e_1, \ldots, e_k\} \quad \text{and} \quad F_k := \Span\{f_1, \ldots, f_k\} \ . \]
Then both $E_k$ and $F_k$ are isotropic, hence their orthogonal complements are coisotropic:
			\[E_k^\perp = \Span\{e_1, \ldots, e_q, l_1, \ldots, l_{p-q}, f_q, \ldots, f_{k+1}\} \ ;\]
			\[F_k^\perp = \Span\{f_1, \ldots, f_q, l_1,\ldots, l_{p-q}, e_q, \ldots, e_{k+1}\} \ . \]

		\begin{defn}
			Define $\O(p, q)$ as the orthogonal group of $\bR^{p, q}$ with respect to the bilinear form $\ob$. Define $\SO(p, q) := \{A \in \O(p, q): \det(A) = 1\}$ and $\PO(p, q) := \O(p, q) / \{\pm I_{p+q}\}$. Denote $\SO_0(p, q)$ (resp. $\PO_0(p, q)$) as identity component of $\SO(p, q)$ (resp. $\PO(p, q)$).
		\end{defn}

		Fix an ordered Witt basis, with respect to which we denote by $B$ the matrix representing $\ob$. Denote by $J_q$ the $q$ by $q$ anti-diagonal matrix with all the anti-diagonal entries to be $1$. Then the Lie algebra of $G = \PO_0(p, q)$ is
		\[\fg = \so(p, q) := \{X \in \gl(\bR^{p, q}): X^T B + B X = 0\} \ ,\text{ where } B = \begin{pmatrix} 0 &0 &J_q\\ 0 &I_{p-q} &0\\ J_q &0 &0\end{pmatrix}\ .\]
Set $K$ as the maximal compact subgroup with Lie algebra 
	\[\fk = \{X\in \so(p, q): X^T = - X\} \simeq \so(p) \oplus \so(q)\ .\]

The standard Cartan subspace is 
		\[\fa = \l\{ H(t) = \opnm{diag} (t_1, \ldots, t_q, 0, \ldots, 0, -t_q, \ldots, -t_1): t_1, \ldots, t_q \in \bR \r\} \ .\]
Define $\varep_i \in \fa^*$ such that $\varep_i(H(t)) = t_i$ for $i = 1, \ldots, q$. The restricted root system is of type $B_q$:
		\[\Sigma = \{ \pm (\varep_i + \varep_j): 1 \leq i < j \leq q\} \cup  \{ \pm (\varep_i - \varep_j): 1 \leq i < j \leq q\} \cup \{\pm \varep_i: 1\leq i \leq q\}\ .\]
A positive system is
		\[\Sigma^+ = \{ \varep_i - \varep_j: 1\leq i < j \leq q\} \cup \{ \varep_i + \varep_j: 1\leq i< j \leq q\} \cup \{\varep_i: 1 \leq i \leq q\}\ .\]
The simple restricted roots are $\al_i := \varep_i - \varep_{i+1}$ for $1\leq i < q$ and $\al_q := \varep_q$. Hence 
		\[\Over{\fa}^+ = \{H(t): t_1 \geq t_2 \geq \cdots \geq t_q \geq 0\}\ .\]

		For $\Th = \{\al_k\}$ for $1\leq k \leq q$, we say $P_k = \Stab_G (E_k)$ and $P_k^- = \Stab_G(F_k)$. Therefore we get isomorphisms
		\[G/P_k \simeq (\IsGr_k \bR^{p, q}, E_k)\ , \quad G/ P_k^- \simeq (\IsGr_k \bR^{p, q}, F_k) \ .\]
Two isotropic planes $V\in \IsGr_i \bR^{p, q}$ and $W \in \IsGr_j \bR^{p, q}$, where $1\leq i, j \leq q$, are compatible (resp. transverse) if and only if $V \subset W$ or $W \subset V$ (resp. $V \oplus W^\perp = \bR^{p, q}$).

		\subsection{$P_\Th$-Anosov representations}
		We use the characterization of $P_\Th$-Anosov representations taken from \cite{GGKW}.

		\begin{theo}[{\cite[Th.~1.3]{GGKW}}] \label{theo 1 Anosov}
			Let $\Ga$ be a word hyperbolic group, $G$ a real reductive Lie group, and $\Th \subset \De$ a nonempty subset of the simple restricted roots of $G$. For any representation $\rho: \Ga \rtar G$, the following conditions are equivalent:			
			\begin{itemize}
				\item[(1)] $\rho$ is $P_\Th$-Anosov;
				
				\item[(2)] There exist continuous, $\rho$-equivariant, dynamics-preserving, and transverse maps $\xi^+: \pa_\infty \Ga \rtar G/ P_\Th$ and $\xi^-: \pa_\infty \Ga \rtar G/ P_\Th^-$, and constants $c, C > 0$ such that $\inp{\al, \mu(\rho(\ga))} \geq c \v{\ga} - C$ for all $\al \in \Th$ and $\ga \in \Ga$.
			
				\item[(3)] There exist $\kappa, \kappa' > 0$ such that for any $\al \in \Sigma_{\Th}^+$ and any geodesic ray $(\ga_n)_{n\in \bbN}$ with $\ga_0 = e$ in the Cayley graph of $\Ga$, the sequence $(\inp{\al, \mu(\rho(\ga_n))})_{n\in \bbN}$ satisfies
				\[\inp{\al, \mu(\rho(\ga_{n+m}))} - \inp{\al, \mu(\rho(\ga_n))} \geq \kappa m - \kappa'\ .\]
			\end{itemize}
		\end{theo}

		\begin{lem}[{\cite[Lem.~8]{Se}}] \label{Selberg}
			A finitely generated group of matrices over a field of characteristic $0$ has a torsion-free subgroup of finite index.
		\end{lem}
		
		Our main results concern restrictions on virtual cohomological dimension, which is defined as follows.
		\begin{defn}[\cite{BM}] \label{vcd}
		The \textit{integral cohomological dimension} of a group $\Ga$ is defined as
			\[\opnm{cd} \Ga  = \sup \l\{n: H^n(\Ga, M) \neq 0 \text{ for some } \bbZ[\Ga] \text{-module } M \r\}\ .\]
If $\Ga$ contains a torsion-free subgroup of finite index, its \textit{virtual cohomological dimension} is defined to be the integral cohomological dimension of any such subgroup, denoted by $\opnm{vcd}(\Ga)$.
		\end{defn}

		In \cite{GW} the authors established a topological obstruction for representations into $\PO(p, q)$ to be Anosov. 
		\begin{prop}[{\cite[Prop.~8.3]{GW}}] \label{leq p}
			Let $p > q$ be two positive integers and $\rho: \Ga \rtar \PO(p, q)$ a morphism. If $\rho$ is $P_1$-Anosov or $P_q$-Anosov, then $\opnm{vcd} (\Ga) \leq p$.
		\end{prop}
		
		\paragraph{Hyperconvexity.} The notion of hyperconvexity was introduced independently by Pozzetti--Sambarino--Wienhard \cite{PSW} and Zhang-Zimmer \cite{ZZ}, building on ideas from \cite{Lab}. 				
		\begin{defn}[{\cite[Def.~6.1]{PSW}}] \label{Defn hyper}
			We say a $\{P_1, P_p\}$-Anosov representation $\rho: \Ga \rtar \PGL_d(\bR)$ is \textit{$(1, 1, p)$-hyperconvex} if, for every triple of pairwise distinct points $x, y, z \in \pa_\infty \Ga$, the sum
			\[\l( \xi^1(x) \oplus \xi^1(y) \r) \cap \xi^{d-p}(z) = \{0\}.\]
		\end{defn}
		
		\begin{prop}[{\cite[Prop.~7.4]{PSW}}] \label{Prop hyper}
			Let $\rho: \Ga \rtar \PGL_d \bR$ be an $(1, 1, p)$-hyperconvex representation and assume that $\pa \Ga$ is a $(p-1)$-dimensional topologically sphere, then $\xi^1(\pa \Ga)$ is a $\calC^1$ manifold with $T_{\xi^1(x)} \xi^1(\pa \Ga) = T_{\xi^1(x)} \bP(\xi^p(x))$.
		\end{prop}

	\subsection{Compact quotient of homogeneous space}
	In this section, we introduce the result in \cite{KT} relating compact quotient to Anosov representation, and introduce some definitions and results about reductive homogeneous space, then finally we prove Lemma~\ref{connected G/H}. 	
	
	\subsubsection{Compact quotient and Anosov}

		\begin{theo}[\cite{KT}, Th. 1.10] \label{theo 1.10 KT}
			Let $G$ be a connected real linear reductive Lie group, $\mu: G \rtar \fa^+$ a Cartan projection, and $H$ a closed subgroup of $G$ such that $\mu(H) = \fa^+ \cap \bigcup\li_{\al \in \Th} \Ker (\al)$ for some subset $\Th \subset \De$ of the simple restricted roots of $G$. Then for any discrete subgroup $\Ga$ of $G$, the following conditions are equivalent:
			\begin{itemize}
				\item[(1)] $\Ga$ acts properly discontinuously and cocompactly on $G/ H$,
				\item[(2)] $\Ga$ is word hyperbolic with 
				\[\opnm{vcd}(\Ga) = \dim(G/ H) - \dim (K / K_H)\ ,\]
where $K_H = K \cap H$, and the natural inclusion $\Ga \hrtar G$ is a $P_\Th$-Anosov representation.
			\end{itemize}
		\end{theo}

		 We apply this theorem to the pseudo-Riemannian hyperbolic space $\mbf{H}^{p, q-1}$ defined in the introduction, where $p > q$. In this case, $G = \PO(p, q)$ and $H = \bP(\O(p, q-1) \times \O(1))$, $\Th = \{\al_q\}$ and $\mu(H) = \fa^+ \cap \Ker (\al_q)$. Therefore, a discrete subgroup $\Ga< \PO(p, q)$ acts properly discontinuously and cocompactly on $\mbf{H}^{p, q-1}$ is equivalent to, $\Ga$ is $P_q$-Anosov with $\opnm{vcd}(\Ga) = p$. This is the Kassel--Tholozan criterion introduced in the introduction.

	\subsubsection{Reductive homogeneous space}
		 
		 \begin{defn}
		Let $G$ be a real linear reductive Lie group, we say a closed subgroup $H$ is a \textit{reductive subgroup} if there exists a Cartan involution $\th: G \rtar G$ such that $\th(H) = H$.

		In this case, the quotient space $G/H$, equipped with the left action of $G$, is called a \textit{reductive homogeneous space} (or \textit{homogeneous space of reductive type}.)
		\end{defn}
		
		Assume that $H$ is a reductive subgroup of $(G, \th)$, then $(H, \th|_H)$ is reductive. $K_H := H^{\th|_H} = K \cap H$ is a maximal compact subgroup of $H$. Let $\fh$ be the Lie algebra of $H$, then we have the Cartan decomposition $\fh = (\fh \cap \fk) \oplus (\fh \cap \fp)$.

	We fix a $G$-invariant and $\th$-invariant non-degenerate symmetric bilinear form $B$ on $\fg$ for which the symmetric bilinear form
	\[B_\th: \fg \times \fg \rtar \bR\ , \quad (X, Y) \mapsto - B(\th (X), Y)\]
is positive definite. By definition, $B_\th$ is $K$-invariant. If $G$ is semisimple, the Killing form satisfies these requirements.	
	
	With respect to $B$ or $B_\th$, the decomposition $\fg = \fk \oplus \fp$ is orthogonal. Let $\fq$ be the orthogonal complement of $\fh$ in $\fg$ with respect to $B$ (also, with respect to $B_\th$). We have the following orthogonal $K_H$-invariant decomposition:
	\[\fg = (\fh \cap \fk) \oplus (\fh \cap \fp) \oplus (\fq \cap \fk) \oplus (\fq \cap \fp) \ .\]
			
		We recall a classical model for the reductive homogeneous space $G/H$ (see, for example, \cite{KMT}). This model was well-known for reductive symmetric space and was proved by Kobayashi in \cite{Ko89} for the general reductive setting.  Define $N := K \times_{K_H} (\fp \cap \fq) = K \times (\fp \cap \fq) / \sim$ where $(k_1, X) \sim (k_2, Y)$ if and only if there exists $k_h \in K_H$ such that $k_1 = k_2 \cdot k_h$ and $X = \Ad(k_h\inv) Y$. By definition, $N$ is a vector bundle over $K / K_H$ with fibers isomorphic to $\fp \cap \fq$.

	\begin{theo}\label{dec}
		The map
		\[\Phi: N = K \times_{K_H} (\fp \cap \fq) \rtar G/H, \quad [k, v] \mapsto k \exp(v) \cdot H\]
is a $K$-equivariant diffeomorphism.
	\end{theo}

	\subsubsection{Proof of Lemma~\ref{connected G/H}} \label{tech lem}
		We prove now Lemma~\ref{connected G/H} by using the model introduced above and some results in geometric group theory.

			\begin{defn}[{\cite[Def.~9.3]{GGT}}]
			Let $Y$ be a proper geodesic metric space. \textit{The number of ends} of $Y$ is the supremum, taken over all compact subsets $K \subset Y$, of cardinalities of $\pi_0(U_K)$, where $U_K$ is the union of unbounded path-connected components of $Y \setminus K$.

		\end{defn}
		
		\begin{lem}[{\cite[Coro.~11.137]{GGT}}] \label{fact con}
		 An infinite word hyperbolic group $\Ga$ is one-ended (the number of ends of $\Ga$ is $1$) if and only if $\partial_\infty \Ga$  is connected.
		 \end{lem}

		\begin{lem}[{\cite[Lem.~9.1]{GGT}}] \label{number}
		The number of ends is a quasi-isometry invariant.
		\end{lem}

		\begin{lem}\label{one-ended}
		Let $G/ H$ be a connected reductive homogeneous spaces and
		\[ \dim(G/H) - \dim(K/K_H) > 1\ ,\]
then $G/H$ is one-ended.
		\end{lem}

		\begin{proof}
		Assume that $G/H$ is a reductive homogeneous space,  $K$ and $K_H$ are maximal compact subgroups of $G$ and $H$ respectively. By Theorem~\ref{dec}, the map
		\[\Phi: N = K \times_{K_H} (\fp \cap \fq) \rtar G/ H, \qquad [k, v] \mapsto k \exp(v) \cdot H\]
is a diffeomorphism, from which we see directly that
		\[\dim (\fp \cap \fq) = \dim (G/H) - \dim (K/K_H) \ .\]

		Denote $V:= \fp \cap \fq$ and $\V{\cdot}$ the norm on $V$ induced by $B_\th$ defined in Section~\ref{Pre}. Note that $\V{\cdot}$ is $K$-invariant since $B_\th$ is $K$-invariant.  
		
		Define $D_R := K \times_{K_H} \{v\in V: \V{v} \leq R\}$, and the $K$-invariance of $\V{\cdot}$ shows that $D_R$ is well-defined. Since both $K$ and $\{v\in V: \V{v} \leq R\}$ are compact, $D_R$ is compact. Moreover, $\{D_R\}_{R=1}^\infty$ is an exhausted sequence of compact subset of $G/H$. Thus, it suffices to prove that, for any $R>0$,
		\[N\setminus D_R = K \times_{K_H} \{v\in V: \V{v} > R\}\]
 is connected (i.e. $\pi_0(N \setminus D_R) = 1$).
 
 		Since the fiber $\{v\in V: \V{v} > R\}$ is homotopy equivalent to a sphere and $\dim V > 1$, we see the fibers are connected. The base space $K/K_H$ is connected since $K/K_H$ and $G/H$ have the same homotopy type and $G/H$ is connected by assumption. Recall that the total space of a fiber bundle with connected base space and connected fiber is connected, and this completes the proof.
		\end{proof}

		\noindent\textbf{Lemma \ref{connected G/H}.} Let $G/ H$ be a connected reductive homogeneous spaces, $K_H$ a maximal compact subgroup of $H$ and $K$ a maximal compact subgroup of $G$ containing $K_H$. Assume that 
			\[\dim(G/H) - \dim(K/K_H) > 1\ .\]
If a word hyperbolic subgroup $\Ga < G$ defines a compact quotient of $G/ H$, then the Gromov boundary $\pa_\infty \Ga$ of $\Ga$ is connected.

		\begin{proof}
		Assume that $\Ga$ acts properly discontinuously and cocompactly on $G/H$. Moreover, we assume that $\Ga$ is torsion-free by \hyperref[Selberg]{Selberg's lemma}.
		
		In general the natural pseudo-Riemannian metric on the reductive symmetric space $G/H$ does not define a positive-definite length metric, so  Milnor-\caron{S}varc Lemma cannot be applied directly to it.  However, $M:=\Ga \backslash G/H$ is a compact manifold, we endow $M$ a complete Riemannian metric $g_M$. Denote the projection by $\pi: G/H \rtar M$, the pull-back $\pi^* g_M$ endows $G/H$ a complete Riemannian structure which is $\Ga$-invariant. We denote $(G/H, \pi^* g_M)$ by $(X, g_X)$. Since $g_X$ is a complete Riemannian metric, Hopf-Rinow theorem shows that $(X, d_X)$ is proper where $d_X$ is the distance induced by $g_X$.
		
		By Milnor-\caron{S}varc lemma, since $\Ga$ acts properly discontinuously, cocompactly and by isometries on the proper length space $(X, d_X)$, the Cayley graph of $\Ga$ is quasi-isometric to $(X, d_X)$. The distance $d_X$ on $X = G/H$ induces the same topology, hence by Lemma~\ref{one-ended}, the number of ends of $X$ is exactly $1$, when $\dim(G/H) - \dim(K/K_H) > 1$. Since the Cayley graph of $\Ga$ is quasi-isometric to $(X, d_X)$, then it is also one-ended. Lemma~\ref{fact con} completes the proof.
		\end{proof}


		\subsection{The pseudo-Riemannian hyperbolic space $\mbf{H}^{p, q-1}$}
		We give a quick introduction to pseudo-Riemannian hyperbolic space $\mbf{H}^{p, q-1}$ and $\mbf{H}^{p, q-1}$-convex cocompactness with reference \cite{DGK1, DGK2, SST, BK, HDRMD}. Recall that $\bR^{p, q}$ is $\bR^{p+q}$ equipped with a bilinear form $\ob$ of signature $(p, q)$, and we define the pseudo-Riemannian hyperbolic space
		\[\mbf{H}^{p, q-1} = \bP \l( \l\{ v\in \bR^{p, q} : \ob(v, v) < 0 \r\} \r)\]
as a subset of $\bP(\bR^{p+q})$.   $\mbf{H}^{p, q-1}$ has a natural pseudo-Riemannian structure of signature $(p, q-1)$ with isometry group $\PO(p, q)$ induced by $\ob$ (see,  for example, \cite[Sec.~2.1]{DGK1}). Thus, $\mbf{H}^{p, q-1}$ is a homogeneous space $G/H$ where $G = \PO(p, q)$ and $H = \bP (\O(p, q-1) \times \O(1))$ mentioned in the introduction.

Define the space of oriented negative lines in  $\bR^{p, q}$
	\[\hat{\mbf{H}}^{p, q-1} = \l\{ v\in \bR^{p, q} : \ob(v, v) < 0 \r\} / \bR_{>0}\ \]
as a subset of $\bP_+(\bR^{p+q})$ the space of oriented lines. Recall the boundary at infinity of $\mbf{H}^{p, q-1}$ in $\bP(\bR^{p+q})$, and of $\hat{\mbf{H}}^{p, q-1}$ in  $\bP_+(\bR^{p+q})$:
		\begin{align*}
		\Ein^{p-1, q-1} := \pa_\infty \mbf{H}^{p, q-1} &=  \bP \l( \l\{ v\in \bR^{p, q} \setminus \{0\} : \ob(v, v) = 0 \r\} \r)\ , \\
		\hat{\Ein}^{p-1, q-1}:= \pa_\infty \hat{\mbf{H}}^{p, q-1} &= \l\{ v\in \bR^{p, q} \setminus \{0\} : \ob(v, v) = 0\r\} / \bR_{> 0}\ .
		\end{align*}
The boundary at infinity of $\mbf{H}^{p, q-1}$ is also called \textit{Einstein universe}. We remark that $\hat{\mbf{H}}^{p, q-1}$ (resp. $\hat{\Ein}^{p-1, q-1}$) is a double cover of $\mbf{H}^{p, q-1}$ (resp. $\Ein^{p-1, q-1}$) (see,  for example, \cite[Sec.~2.3]{BK}). The Einstein universe $\Ein^{p-1, q-1}$ carries a natural conformal structure of signature $(p-1, q-1)$, and it induces a conformal structure on $\hat{\Ein}^{p-1, q-1}$. Indeed, for any $[v] \in \Ein^{p-1, q-1}$, we have the following identifications
		\[T_{[v]} \Ein^{p-1, q-1} \cong \Hom([v], [v]^\perp/ [v]) \cong [v]^* \otimes_\bR \l ([v]^\perp/ [v] \r) \ .\]
Any choice of non-zero vector in $[v]$ identify $T_{[v]} \Ein^{p-1, q-1}$ with $[v]^\perp / [v]$, thus defining a quadratic form on $T_{[v]} \Ein^{p-1, q-1}$. Since $[v]$ contains no preferred vector, this construction only yields a conformal structure on $T_{[v]} \Ein^{p-1, q-1}$.

	\begin{prop}[{\cite[Prop.~1.11]{KiTr}}] \label{conformal}
		$\hat{\Ein}^{p-1, q-1}$ is conformally equivalent to 
		\[\l(\mbf{S}^{p-1} \times \mbf{S}^{q-1}, g_{\mbf{S}^{p-1}} \oplus \l(-g_{\mbf{S}^{q-1}}\r) \r)\ ,\]
where $g_{\mbf{S}^{k}}$ is the round metric on $\mbf{S}^k$.
	\end{prop}

	\begin{defn}[{\cite[Def. 3.1]{DGK1}}]
		A subset $\t{\La}$ of $\bR^{p, q}$ is \textit{negative} if $\ob(\t{x}, \t{y}) < 0$ for all $\t{x}, \t{y} \in \t{\La}$ with $\t{x} \neq \t{y}$. A subset of $\Ein^{p-1, q-1}$ (resp. $\hat{\Ein}^{p-1, q-1}$) is \textit{negative} if it is the projection to $\Ein^{p-1, q-1}$ (resp. $\hat{\Ein}^{p-1, q-1}$) of a negative subset of $\bR^{p, q}$.
	\end{defn}
		
	The negativity of a subset  naturally induces a strictly $1$-Lipschitz map between spheres. For future reference, we recall the following result from \cite{BK} (see also \cite[Lem.~4.7]{SST}, where such subsets are termed ``positive'' instead).
	
	\begin{prop}[{\cite[Lem. 3.3]{BK}}] \label{Lem 3.3 BK}
		Fix a splitting $\hat{\Ein}^{p-1, q-1} \simeq \mbf{S}^{p-1} \times \mbf{S}^{q-1}$ defined by the choice of a timelike $q$-plane of $\bR^{p, q}$. A subset $\hat{\La}$ of $\hat{\Ein}^{p-1, q-1}$ is negative if and only if, in this splitting, it is the graph of a strictly $1$-Lipschitz map $f: A \rtar \mbf{S}^{q-1}$ for some subset $A$ of $\mbf{S}^{p-1}$.
	\end{prop}

	\paragraph{The $\mbf{H}^{p, q-1}$-convex cocompactness.} A subset of projective space is said to be \textit{convex} if it is contained and convex in some affine chart. A subset of projective space is said to be \textit{properly convex} if its closure is convex.

	\begin{defn}[{\cite[Def.~1.2]{DGK1}}] \label{cc}
		A discrete subgroup $\Ga$ of $\PO(p, q)$ is \textit{$\mbf{H}^{p, q-1}$-convex cocompact} if it acts properly discontinuously and cocompactly on some properly convex closed subset $\calC$ of $\mbf{H}^{p, q-1}$ with nonempty interior whose ideal boundary $\pa_i \calC := \Over{\calC} \setminus \calC$ does not contain any nontrivial projective segment.
	\end{defn}

	\begin{defn}[{\cite[Def.~2.3]{DGK1}}]
			Let $\Ga$ be a discrete subgroup of $\PGL(V)$. The \textit{proximal limit set of $\Ga$ in $\bP(V)$} is the closure $\La_\Ga$ of the set of attracting fixed points of elements of $\Ga$ which are proximal in $\bP(V)$.
	\end{defn}

	\begin{theo}[{\cite[Th.~1.24]{DGK2}}] \label{cc-Anosov}
		Let $p, q \geq 2$. For an infinite discrete subgroup $\Ga$ of $\PO(p, q)$, the following are equivalent:
		\begin{itemize}
			\item[(1)] $\Ga$ is $\mbf{H}^{p, q-1}$-convex cocompact.
			\item[(2)] $\Ga$ is word hyperbolic, the inclusion $\Ga \hrtar \PO(p, q)$ is $P_1^{p, q}$-Anosov, and the proximal limit set $\La_\Ga \subset \Ein^{p-1, q-1}$ is negative.
		\end{itemize}
	\end{theo}
	\begin{rem}
	When $\Ga$ is $\mbf{H}^{p, q-1}$-convex cocompact, the induced boundary map is $\xi^1: \pa_\infty \Ga \rtar \Ein^{p-1, q-1} = G/ P^{p, q}_1$ and we have $\xi^1(\pa_\infty \Ga) = \La_\Ga \subset \Ein^{p-1, q-1}$ (see, for example,  \cite[Th.~1.15]{DGK2}).
	\end{rem}
	
	We cite a theorem which gives restrictions on the Gromov boundary $\pa_\infty \Ga$ of an $\mbf{H}^{p, q-1}$-convex cocompact subgroup.
	\begin{theo}[{\cite[Coro.~11.10]{DGK2}}] \label{res}
		Let $p, q\geq 1$.
		\begin{itemize}
			\item[(1)] Let $\Ga$ be an infinite discrete subgroup of $\PO(p, q)$ which is $\mbf{H}^{p, q-1}$-convex cocompact. Then $\Ga$ is word hyperbolic and has virtual cohomological dimension at most $p$. Moreover, the virtual cohomological dimension of $\Ga$ is $p$ if and only if the Gromov boundary $\pa_\infty \Ga$ is homeomorphic to a $(p-1)$-sphere.
			\item[(2)] Let $\Ga$ be an infinite word hyperbolic group with connected Gromov boundary $\pa_\infty \Ga$ and virtual cohomological dimension larger than $q$. If there is a $P_1^{p, q}$-Anosov representation $\rho: \Ga \rtar \PO(p, q)$, then $p > q$ and the group $\rho(\Ga)$ is $\mbf{H}^{p, q-1}$-convex cocompact.
		\end{itemize}
	\end{theo}


	\section{Topological obstruction for Anosov representations}
	In this section, we prove our main theorem. Due to Proposition~\ref{leq p}, it suffices to prove that $\opnm{vcd}(\Ga) \neq p$. 
			
		\subsection{Hyperconvexity and $\calC^1$ Lemma for $\PO(p, q)$}

		The crucial observation is the $\calC^1$-regularity of the limit set:
				
		\begin{prop} \label{is cc}  \label{Prop hyper popq} \label{lem C1}
			Let $p > q$ be two integers, $\Ga$ a discrete group with $\opnm{vcd} = p$ and $\rho: \Ga \rtar \PO(p, q)$ a morphism. If $\rho$ is $\{P_1, P_q\}$-Anosov, then $\rho$ is $\mbf{H}^{p, q-1}$-convex cocompact, the limit set $\La_\rho = \xi_1(\pa \Ga)$ in $\Ein^{p-1, q-1}$ is a $\calC^1$ negative $(p-1)$-sphere and $T_z \La_\rho$ contains a $(q-1)$-dimensional isotropic subspace for any $z\in \La_\rho$.
		\end{prop}

		We observe that for a $\{P_1, P_q\}$-Anosov representation $\rho: \Ga \rtar \PO(p, q)$, $\rho$ is automatically $(1, 1, p)$-hyperconvex.
		\begin{lem} \label{lem C1 popq}
			For $p > q$, if $\rho: \Ga \rtar \PO(p, q)$ is $\{P_1, P_q\}$-Anosov, then $\rho$ is $(1, 1, p)$-hyperconvex. 
		\end{lem}
		
		\begin{proof}
			By natural inclusion $\iota: \PO(p, q) \hrtar \PGL_{p+q} \bR$, the representation $\iota \circ \rho: \Ga \rtar \PGL_{p+q} \bR$ is $\{P_1, P_q\}$-Anosov. The limit maps are $\xi^k: \pa_\infty \Ga \rtar \IsGr_k \bR^{p, q}$ for $k = 1, q$.
			
			By definition it suffices to prove $(\xi^1(x) \oplus \xi^1(y)) \cap \xi^q (z) = \{0\}$ for any $x, y, z \in \pa_\infty \Ga$ pairwise distinct. For any nonzero vectors $u\in \xi^1(x)$ and $v \in \xi^1(y)$, by transversality we have $\ob(u, v)\neq 0$. Note that $\xi^q(z)$ is isotropic, and $\ob(au + bv, au+ bv) = 2ab \ob(u, v)$, then $au + bv \in \xi^q(z)$ only if $a= 0$ or $b = 0$, both are impossible by the transversality of boundary maps.
		\end{proof}

		\begin{proof}[Proof of Proposition~\ref{is cc}]		
			By Kassel--Tholozan criterion, $\Ga$ defines a compact quotient of $\mbf{H}^{p, q-1}$, hence by Corollary~\ref{connected}, the Gromov boundary $\pa_\infty \Ga$ is connected. By Theorem~\ref{res}, we  say $\rho$ is $\mbf{H}^{p, q-1}$-convex cocompact and the limit set $\La_\rho \subset \Ein^{p-1, q-1}$ is negative and homeomorphic to a $(p-1)$-sphere. The $\calC^1$-regularity comes from Lemma~\ref{lem C1 popq} and Proposition~\ref{Prop hyper}. The $(q-1)$-dimensional isotropic subspace is given by the boundary map $\xi^q$.
		\end{proof}

	\subsection{Proof of Main Theorem and Corollary~\ref{coro A}}
	
	\begin{proof}[Proof of Main Theorem]
		Let $p > q \geq 2$ and $\Ga$ be a word hyperbolic group. Assume $\rho: \Ga \rtar \PO(p, q)$ is $\{P_1, P_q\}$-Anosov with $\opnm{vcd}(\Ga) = p$.  By Proposition~\ref{is cc}, $\Ga$ is $\mbf{H}^{p, q-1}$-convex cocompact and the limit set $\La_\rho = \xi^1 (\pa_\infty \Ga) \subset \Ein^{p-1, q-1}$ is a $\calC^1$ negative $(p-1)$-sphere and $T_z \La_\rho$ contains a $(q-1)$-dimensional isotropic subspace  for any $z\in \La_\rho$.
				
		Recall that the projection $\pi: \hat{\Ein}^{p-1, q-1} \rtar \Ein^{p-1, q-1}$ is a double cover, since $\La_\rho \simeq \mbf{S}^{p-1}$ is simply-connected, the restricted covering $\pi\inv(\La_\rho) \rtar \La_\Ga$ is trivial. Hence $\pi\inv(\La_\rho)$ is a disjoint union $\hat{\La}_+ \sqcup \hat{\La}_-$ where each component maps homeomorphically to $\La_\rho$. We denote $\hat{\La} := \hat{\La}_+ \subset \hat{\Ein}^{p-1, q-1}$ which is negative since $\La_\rho$ is negative. By Proposition~\ref{Lem 3.3 BK} and fixing a splitting $\varphi: \hat{\Ein}^{p-1, q-1} \simeq \mbf{S}^{p-1} \times \mbf{S}^{q-1}$,  there is a strictly $1$-Lipschitz map $f: A \rtar \mbf{S}^{q-1}$, where $A \subset \mbf{S}^{p-1}$, such that $\hat{\La} = \Graph f$.

	Since $\hat{\La}$ is homeomorphic to both $A$ and $\mbf{S}^{p-1}$, then $A\subset \mbf{S}^{p-1}$ is homeomorphic to $\mbf{S}^{p-1}$. We claim that $A = \mbf{S}^{p-1}$. Indeed, denote $h: \mbf{S}^{p-1} \rtar A$ the homeomorphism, by the invariance of domain, $A = h(\mbf{S}^{p-1})$ is open. Since $\mbf{S}^{p-1}$ is compact, then $A = h(\mbf{S}^{p-1})$ is also compact. Now $A \subset \mbf{S}^{p-1}$ is both open and closed, by the connectedness of $\mbf{S}^{p-1}$, $A =\mbf{S}^{p-1}$. Since $\hat{\Ga} = \Graph f$ is $\calC^1$, then so is $f: \mbf{S}^{p-1} \rtar \mbf{S}^{q-1}$.
				
		For any $x\in \pa_\infty  \Ga$, we say $\xi^1(x) \in \La_\Ga$ and lift it to $\hat{\xi}^1(x) \in \hat{\La}$, then write $\varphi (\hat{\xi}^1(x)) = (u, u') \in \mbf{S}^{p-1} \times \mbf{S}^{q-1}$ where $u' = f(u)$. We have the decomposition
	\[T_{(u, u')} \l(\mbf{S}^{p-1} \times \mbf{S}^{q-1}\r) = T_u \mbf{S}^{p-1} \oplus T_{u'} \mbf{S}^{q-1}\ .\]
For any $v \in T_{(u, u')} \Graph f$, we write $v = (v_1, v_2)$ where $v_1 \in T_u \mbf{S}^{p-1}$ and $v_2 = (T_u f) (v_1) \in T_{u'} \mbf{S}^{q-1}$. By Proposition~\ref{conformal}, the bilinear form $\ob$ is proportional to $\og_{\mbf{S}^{p-1}} \oplus \l(- \og_{\mbf{S}^{q-1}}\r)$, i.e. for any $v, w \in T_{(u, u')} \Graph f$, $\ob(v, w)$  is proportional to 
	\[\og_{\mbf{S}^{p-1}}(v_1, w_1) - \og_{\mbf{S}^{q-1}}(v_2, w_2) =  \og_{\mbf{S}^{p-1}}(v_1, w_1) - \og_{\mbf{S}^{q-1}}\l((T_u f) (v_1), (T_u f) (w_1) \r)\ .\]
		
	Since by Proposition~\ref{Prop hyper popq}, for any $x\in \pa_\infty \Ga$, $T_{\xi^1(x)} \La_\Ga$ contains an isotropic $(q-1)$-plane, and so does $T_{\hat{\xi}^1(x)} \hat{\La}$. Thus, we can find a $(q-1)$-subspace $W \subset T_u \mbf{S}^{p-1}$ such that $T_u f$ restricts to an isometry on $W$. We conclude that $f: \mbf{S}^{p-1} \rtar \mbf{S}^{q-1}$ is a submersion hence an open map. Note that $\mbf{S}^{p-1}$ is compact, then $f(\mbf{S}^{p-1})$ is compact hence closed in $\mbf{S}^{q-1}$. Therefore, $f$ is surjective. 
	
	Since $f$ is strictly $1$-Lipschitz, using $B(x, r)$ to denote a $r$-ball around $x$, 
	\[f\l(\mbf{S}^{p-1}\r) =  f\l(\overline{B}(x_0, \pi) \r)\subsetneqq \overline{B} \l(f(x_0), \pi \r) = \mbf{S}^{q-1}\ ,\] 
which yields a contradiction.	
	\end{proof}

	\begin{proof}[Main Theorem implies Corollary~\ref{coro A}]
			Let $\Ga$ be a discrete subgroup of $\PO(p, q)$. If $\Ga$ is $\mbf{H}^{p, q-1}$-convex cocompact, by Theorem~\ref{cc-Anosov}, $\Ga$ is $P_1$-Anosov. If $\Ga$ acts properly discontinuously and cocompactly on $\mbf{H}^{p, q-1}$, by Kassel--Tholozan criterion, $\Ga$ is $P_q$-Anosov and $\opnm{vcd}(\Ga) = p$. Thus, this corollary follows directly from our main theorem.
		\end{proof}


		\bibliographystyle{alpha}
		\addcontentsline{toc}{section}{References}
		\bibliography{refs}

	\end{document}